\documentclass{amsart}
\usepackage{amsthm,enumerate}
\usepackage{amsmath,amssymb}
\newtheorem{theorem}{Theorem}

\newtheorem{lem}[theorem]{Lemma}

\numberwithin{theorem}{section}

\allowdisplaybreaks
\newcommand{\Om} {\Omega}
\newcommand{\pa} {\partial}

\newcommand{\al} {\alpha}

\newcommand{\De} {\Delta}
\newcommand{\la} {\lambda}

\def\R{{\mathbb R}}

\def\R{{\mathbb R}}

\newcommand{\na} {\nabla}

\newcommand{\rar}{\rightarrow}
\numberwithin{equation}{section}
\begin{document}
\title[Picone's Identity]{Picone's Identity for $p$-biharmonic operator and Its Applications}
\author[G. Dwivedi ]
{G.\,Dwivedi}
\address{G.\,Dwivedi \hfill\break
 Indian Institute of Technology Gandhinagar \newline
 Vishwakarma Government Engineering College Complex \newline
 Chandkheda, Visat-Gandhinagar Highway, Ahmedabad \newline
 Gujarat, India - 382424}
 \email{dwivedi\_gaurav@iitgn.ac.in}
\thanks{Submitted 19--03--2015.  Published-----.}
\begin{abstract}
In this article we prove the nonlinear analogue of Picone's identity for $p-$biharmonic operator. As an application of our result we show that the Morse index of the zero solution to a $p-$biharmonic boundary value problem is $0$. We also prove a Hardy type  inequality and Sturmian comparison principle. We also show the strict monotonicity of the principle eigenvalue and linear relationship between the solutions of a system of singular  $p$-biharmonic system.
\end{abstract}
\maketitle
\section{Introduction}
The classical Picone's identity says that
for differentiable functions $v>0$ and $u\geq 0,$
\begin{equation}\label{pico}
|\nabla u|^2 + \frac{u^2}{v^2} |\nabla v|^2- 2 \frac{u}{v} \nabla u\nabla v= |\nabla u|^2- \nabla \left( \frac{u^2}{v} \right)\nabla v \geq 0.
\end{equation}
\eqref{pico} has an enormous applications to second-order elliptic equations and systems, see for instance, \cite{al1,al2,al3,manes} and the references therein.
Nonlinear analogue of \eqref{pico} is established by J. Tyagi \cite{Tyagi}. In order to apply \eqref{pico} to p-Laplace equations, \eqref{pico} is extended by W. Allegretto and Y.X.Huang \cite{ale}.
Nonlinear analogue of Picone's type identity for  p-Laplace equations is established by K. Bal \cite{bal}.

In \cite{dun}, D.R.Dunninger established a Picone identity for a class of fourth order elliptic differential inequalities. This identity says that if $u,\,v,\,a\De u,\,A\De v$
are twice continuously differentiable functions with $v(x)\neq 0$ and $a$ and $A$ are positive weights, then
\begin{align}\label{wp1}
&  div\left[ u\nabla(a\De u)- a \De u\nabla u - \frac{u^2}{v} \nabla (A \De v)+ A \De v. \nabla \left(\frac{u^2}{v}  \right)\right]  \nonumber\\
&= - \frac{u^2}{v} \De (A \De v)+ u \De (a \De u) + (A-a)(\De u)^2         \nonumber \\
& -  A \left( \Delta u - \frac{u}{v} \Delta v \right)^{2}+  A  \frac{2\Delta v}{v} \left(\nabla u - \frac{u}{v}\nabla v\right)^{2}.
\end{align}
 With some simplifications in \eqref{wp1}, we obtain the following identity:

 Let $u$ and $v$ be twice continuously differentiable functions in $\Om$ such that $v>0,\,\,-\De v>0$ in $\Om.$ Denote
  \begin{eqnarray}\label{picb}
  L(u,\,v)&=& \left( \Delta u - \frac{u}{v} \Delta v \right)^{2} - \frac{2\Delta v}{v} \left(\nabla u - \frac{u}{v}\nabla v\right)^{2}.\\ \nonumber
 R(u,\,v) &= &|\De u|^{2}- \De\left(\frac{u^2}{v} \right) \De v.
\end{eqnarray}
 Then \rm{(i)}\,\,$L(u,\,v)= R(u,\,v)$\,\,\,\rm{(ii)}\,\,$L(u,\,v)\geq 0$ and \rm{(iii)}\,\,$L(u,\,v)=0$ in $\Om$ if and only if $u= \alpha v $ for some $\alpha \in \R.$

Nonlinear analogue of \eqref{picb} is established by G. Dwivedi and J. Tyagi \cite{dwivedi}. Picone's identity for p-biharmonic operator is established by J. Jaro\v{s} \cite{jar1}.  Picone's identity for
elliptic differential operators are discussed by N. Yoshida \cite{yoshida} and J. Jaro\v{s} \cite{jar2}.

In this article we establish the nonlinear analogue of Picone's identity for p-biharmonic operator. We also discuss some qualitative results in the spirit of
W. Allegretto and Y.X.Huang \cite{ale} and J. Tyagi \cite{Tyagi}.

The plan of the paper is as follows: Section 2 deals with nonlinear analogue of Picone's identity. In section 3, we give several application of Picone's identity to p-biharmonic equations.
\section{Main Results}
Throughout this article, we assume the following hypotheses, unless otherwise stated.
\begin{enumerate}[(i)]
\item $\Delta_p^2:= \Delta(|\Delta u|^{p-2}\Delta u),$ denotes $p-$biharmonic operator.
 \item  $\Omega $ denotes any domain in $\R^n.$
\item  $1<p<\infty.$
\item $f:\R \rightarrow (0,\infty)$ be a $C^2$ function.
\end{enumerate}
First we state Young's inequality, which will be used later.
\begin{lem}\label{young}
If $a$ and $b$ are two nonnegative real numbers and $p$ and $q$ are such that $\frac{1}{p}+\frac{1}{q}=1$, then
\begin{equation}\label{young_eqn}
ab \leq \frac{a^p}{p} + \frac{b^q}{q},
\end{equation}
equality holds if and only if $a^p=b^q$.
\end{lem}
\begin{proof}
For proof we refer to \cite{evans}.
\end{proof}
 Next we give Picone's identity for $p-$biharmonic operator.

\begin{lem}\label{picone_l}
Let $u\geq 0,\, v>0$ be twice continuously differentiable functions in $\Om$ and $-\De v >0$ in $\Om$. Denote
\begin{eqnarray*}
R(u,v) &=& |\De u|^p-\De\left(\frac{u^p}{v^{p-1}}\right)|\De v|^{p-2}\De v\\
L(u,v) &=& |\De u|^p+\frac{(p-1)u^p}{v^p}|\De v|^p-\frac{pu^{p-1}}{v^{p-1}}|\De v|^{p-2}\De u\De v\\
& & -\frac{p(p-1)u^{p-2}}{v^{p-1}}\De v|\De v|^{p-2}\left(\nabla u-\frac{u}{v}\nabla v\right)^2.
\end{eqnarray*}
Then \rm{(i)} $L(u,v)=R(u,v)$, \rm{(ii)} $L(u,v)\geq 0$, \rm{(iii)} $L(u,\,v)=0$ in $\Om$ if and only if $u= \alpha v $ for some $\alpha \in \R.$
\end{lem}
\begin{proof}
Let us expand $R(u,v)$:
\begin{eqnarray*}
R(u,v) &=& |\De u|^p-\De\left(\frac{u^p}{v^{p-1}}\right)|\De v|^{p-2}\De v\\
&=& |\De u|^p+\frac{(p-1)u^p}{v^p}|\De v|^p-\frac{pu^{p-1}}{v^{p-1}}|\De v|^{p-2}\De u\De v\\
& & +\frac{2p(p-1)u^{p-1}}{v^p}|\De v|^{p-2}(\nabla u. \nabla v)\De v-\frac{p(p-1)u^{p-2}}{v^{p-1}}|\na u|^2|\De v|^{p-2}\De v \\
& & -\frac{p(p-1)u^p}{v^{p+1}}|\nabla v|^2|\De v|^{p-2}\De v\\
&=& \underbrace{\left(|\De u|^p+\frac{(p-1)u^p}{v^p}|\De v|^p-\frac{pu^{p-1}}{v^{p-1}}|\De v|^p|\De u|\,|\De v|\right)}_{(\rm{I})} \\
& & \underbrace{\frac{pu^{p-1}}{v^{p-1}}|\De v|^{p-2}\left(|\De u|\, |\De v|-\De u\De v\right)}_{(\rm{II})}+\\
& &\underbrace{\frac{2p(p-1)u^{p-1}}{v^p}|\De v|^{p-2}(\nabla u. \nabla v)\De v-\frac{p(p-1)u^{p-2}}{v^{p-1}}|\na u|^2|\De v|^{p-2}\De v}_{(\rm{III})} \\
& & \underbrace{-\frac{p(p-1)u^p}{v^{p+1}}|\nabla v|^2|\De v|^{p-2}\De v}.\\
\end{eqnarray*}
First consider (\rm{III}):
\begin{eqnarray*}
(\rm{III}) &=& \frac{p(p-1)u^{p-2}}{v^{p-1}}\De v|\De v|^{p-2}\left(\frac{2u}{v}\na u.\na v-|\na u|^2-\frac{u^2}{v^2}|\na |^2\right)\\
&=& -\frac{p(p-1)u^{p-2}}{v^{p-1}}\De v|\De v|^{p-2}\left(\na u-\frac{u}{v}\na v\right)^2.
\end{eqnarray*}
This shows that $R(u,v)=L(u,v)$. \\
Since $u\geq 0,\, v>0,\, -\De v> 0$, we get (\rm{III})$\geq 0$.\\
 Next we consider (\rm{II}). Since $|\De u|\, |\De v|\geq \De u\De v$, therefore, (\rm{II})$\geq 0$.\\
 Now consider (\rm{I}):
 \[{(\rm{I})}=|\De u|^p+\frac{(p-1)u^p}{v^p}|\De v|^p-\frac{pu^{p-1}}{v^{p-1}}|\De v|^p|\De u|\,|\De v|.\]
 By Lemma \ref{young}, with $a=|\De u|$ and $\displaystyle b=\frac{|u^{p-1}| |\De v|^{p-1}}{v^{p-1}}$, we get
 \[\frac{pu^{p-1}}{v^{p-1}}|\De u|\,|\De v|^{p-1} \leq |\De u|^p+\frac{p}{q}\frac{u^{p}}{v^p}|\De v|^p,\]
 which gives
 \[|\De u|^p+\frac{(p-1)u^p}{v^p}|\De v|^p-\frac{pu^{p-1}}{v^{p-1}}|\De v|^p|\De u|\,|\De v|\geq 0.\]
 This proves that (\rm{I})$\geq 0$. This proves the (\rm{ii}), that is, $L(u,v)\geq 0$.\\
 Now $L(u,v)=0$ in $\Om$ implies that $\na u-\frac{u}{v}\na v=0,$ provided $u(x_0)\neq 0$ for some $x_0\in \Om$. This gives $\na\left(\frac{u}{v}\right)=0$, that is, $u=\al v$ for some $\al\in \R$. This completes the proof.
\end{proof}
Now we establish the nonlinear analogue of Picone's identity for $p-$biharmonic operator.
\begin{lem}\label{picone_nl}
Let $u$ and $v$ be twice continuously differentiable functions in $\Om$ such that $u\geq 0$ and $-\De v >0$ in $\Om$. Let $f\colon \R \rar \left( 0, \infty \right)$ be a $C^2$ function
such that $f'(y) \geq (p-1)[f(y)]^{\frac{p-1}{p-2}}$ and $f''(y) \leq 0,\, \forall \, y\in \R$. Denote
\begin{eqnarray*}
L(u,v)&=&|\De u|^p-\frac{pu^{p-1}|\De v|^{p-2}\De u\De v}{f}+\frac{u^pf'(v)|\De v|^p}{f^2}\\
& &+\frac{u^pf''(v)|\nabla v|^2|\De v|^{p-2}\De v}{f^2}\\
& &-\frac{1}{2}\frac{\De v|\De v|^{p-2}u^{p-2}}{f}\left[\left(\frac{2uf'}{f}\nabla v-p\nabla u\right)^2+p(p-1)|\nabla u|^2\right].\\
R(u,v)&=&|\De u|^p-\De \left(\frac{u^p}{f(v)}\right)|\De v|^{p-2}\De v.
\end{eqnarray*}
Then \rm{(i)} $L(u,v)=R(u,v)$, \rm{(ii)} $L(u,v)\geq 0$, \rm{(iii)} $L(u,\,v)=0$ in $\Om$ if and only if $u= \alpha v $ for some $\alpha \in \R.$
\end{lem}
\begin{proof}
Let us expand the $R(u,v):$
\begin{eqnarray*}
R(u,v) & = &|\De u|^p-\De \left(\frac{u^p}{f(v)}\right)|\De v|^{p-2}\De v \\
& = & |\De u|^p+\frac{u^pf'(v)|\De v|^p}{f^2}-\frac{pu^{p-1}|\De v|^{p-2}\De u\De v}{f(v)}\\ & &
-\frac{p(p-1)u^{p-2}|\nabla u|^2 |\De v|^{p-2} \De v}{f}-\frac{2u^pf'^2|\nabla v|^2|\De v|^{p-2}\De v}{f^3}\\& &
+\frac{2pf'(v)u^{p-1}(\nabla u.\nabla v)|\De v|^{p-2}\De v}{f^2}+\frac{u^pf''(v)|\nabla v|^2|\De v|^{p-2}\De v}{f^2}\\
&=& \underbrace{\left(|\De u|^p+\frac{u^pf'(v)|\De v|^p}{f^2}-\frac{pu^{p-1}|\De v|^{p-2}\De u\De v}{f(v)}\right)}_{(\rm{I})} \\
& & -\,\underbrace{\frac{\De v|\De v|^{p-2} u^{p-2}}{f}\left(p(p-1)|\nabla u|^2+2 \frac{u^2f'^2}{f^2}|\nabla v|^2-\frac{2pf'u}{f}(\nabla u.\nabla v)\right)}_{(\rm{II})}\\
& & +\underbrace{\frac{u^pf''(v)|\nabla v|^2|\De v|^{p-2}\De v}{f^2}}_{(\rm{III})}.
\end{eqnarray*}
First we consider $(\rm{II}).$
\begin{eqnarray*}
(\rm{II})&=&-\frac{\De v|\De v|^{p-2}u^{p-2}}{f}\left[p(p-1)|\nabla u|^2+\frac{2u^2f'^2}{f^2}|\nabla v|^2-\frac{2pf'u(\nabla u.\nabla v)}{f}\right.\\
& &+\left(\frac{p\na u}{\sqrt{2}}\right)^2
 -\left.\left(\frac{p\na u}{\sqrt{2}}\right)^2\right]\\
 &=& -\frac{1}{2}\frac{\De v|\De v|^{p-2}u^{p-2}}{f}\left[\left(\frac{2uf'}{f}\nabla v-p\nabla u\right)^2+p(p-1)|\nabla u|^2\right].
 \end{eqnarray*}
 This completes $(\rm{i})$, that is , $L(u,v)=R(u,v).$
Also $(\rm{II})\geq 0$, since $-\De v> 0$.\\
 Next consider $(\rm{I}).$
\begin{eqnarray*}
(\rm{I}) &=&\left(|\De u|^p+\frac{u^pf'(v)|\De v|^p}{f^2}-\frac{pu^{p-1}|\De v|^{p-2}|\De u|\,|\De v|}{f(v)}\right)\\
& &+\frac{pu^{p-1}|\De v|^{p-2}}{f}\left(|\De u|\,|\De v|-\De u \De v\right),
\end{eqnarray*}
clearly second term of above equation is nonnegative.
So on using Young's inequality (Lemma \ref{young}) with  $a=|\De u|$, $\displaystyle b=\frac{(u|\De v|)^{p-1}}{f}$, we get
\[\frac{|\De u|u^{p-1}|\De v|^{p-1}}{f}\leq \frac{|\De u|^p}{p}+\frac{(u|\De v|)^{(p-1)q}}{qf^q}\]
\[p\frac{|\De u|u^{p-1}|\De v|^{p-1}}{f}\leq |\De u|^p+(p-1)\frac{(u|\De v|)^{(p-1)q}}{f^q},\]
equality holds when
\begin{equation}\label{npi1}
|\De u|=\frac{u|\De v|}{[f(v)]^{\frac{q}{p}}}.
\end{equation}
Now on using $f'(y)\geq (p-1)[f(y)]^{\frac{p-2}{p-1}}$, we get
\[|\De u|^p+\frac{(u^pf'(v)|\De v|^p)}{f^2}- \frac{|\De u|u^{p-1}|\De v|^{p-1}}{f}\geq 0\]
and equality holds when
\begin{equation}\label{npi2}
f'(y)=(p-1)[f(y)]^{\frac{p-2}{p-1}}
\end{equation}
This gives $(\rm{I})\geq 0.$\\
Now $(\rm{III})\geq 0$, since  $-\De v> 0$ and $f''(v)\leq 0$. This proves $(\rm{ii})$.
\end{proof}

\section{Applications}
In this section we will give some applications of  nonlinear Picone's identity following the spirit of \cite{ale}.

\noindent\textbf{Hardy type result.} We start with establishing a Hardy type inequality for p-biharmonic operator.
\begin{theorem}\label{hardy_type}
Let there be a $v\in C_c^\infty(\Om)$ such that
\[\De (|\De v|^{p-2}\De v)\geq \la g f(v),\,\, v>0\,\, \text{in}\, \Om,\, -\De v>0\,\, \text{in}\, \Om,\]
for some $\la>0$ and a nonnegative continuous function $g$ then for any $u\in C_c^\infty(\Om)$; $u\geq 0$ it holds that
\begin{equation}\label{hardy}
\int_\Om |\De u|^pdx\geq \lambda\int_\Om g|u|^p dx,
\end{equation}
where $f$ satisfies $f'(y)\geq (p-1)[f(y)]^{\frac{p-2}{p-1}}.$
\end{theorem}
\begin{proof}
Take $\phi\in C_c^\infty(\Om),$ $\phi>0$. By Lemma \ref{picone_nl}, we have
\begin{align*}
0 &\leq \int_\Om L(\phi,v)\, dx \\ & =
\int_\Om R(\phi,v) \,dx=\int_\Om \left(|\De \phi|^p-\De \left(\frac{\phi^p}{f(v)}\right)|\De v|^{p-2}\De v\right)dx \\ &=
\int_\Om |\De \phi|^p dx-\int_\Om \frac{\phi^p}{f(v)}\De (|\De v|^{p-2}\De v)\, dx \\
& \leq \int_\Om |\De \phi|^pdx-\la \int_\Om \phi^p g\, dx
\end{align*}
letting $\phi\rar u$, we get
\[\int_\Om |\De u|^pdx\geq \lambda\int_\Om g|u|^p dx.\]
\end{proof}
\noindent\textbf{Strumium comparison principle.} Comparison principles play vital role in study of partial differential equations. Here, we establish nonlinear version of Sturmium comparison principle for p-biharmonic operator.
\begin{theorem}
Let $f_1$ and $f_2$ are two weight functions such that $f_1<f_2$ and $f$ satisfies $f'(y)\geq (p-1)[f(y)]^{\frac{p-2}{p-1}}$. If there is a positive solution $u$ satisfying
\begin{equation}\label{sturmium_1}
\begin{split}
\De_p^2 v=f_1(x)|u|^{p-2}u\,\,\,\text{in}\, \Om \\
u=0=\Delta u\,\,\, \text{on}\,\, \pa\Om.
\end{split}
\end{equation}
Then any nontrivial solution $v$ of
\begin{equation}\label{sturmium_2}
\begin{split}
\De_p^2 v=f_2(x)f(v)\,\,\,\text{in}\, \Om \\
u=0=\Delta u\,\,\, \text{on}\,\, \pa\Om,
\end{split}
\end{equation}
must change sign.
\end{theorem}
\begin{proof}
Let us assume that there exists a solution $v>0$ of \eqref{sturmium_2} in $\Om$. Then by Picone's identity we have
\begin{align*}
0 & \leq\int_\Om L(u,v)\, dx=\int_\Om R(u,v)\, dx\\
& = \int_\Om|\De u|^p-\De \left(\frac{u^p}{f(v)}\right)|\De v|^{p-2}\De v\, dx\\
& = \int_\Om \left(f_1(x)u^p-f_2(x)u^p\right)dx\\
& = \int_\Om\left(f_1-f_2\right)u^pdx<0,
\end{align*}
which is a contradiction. Hence, $v$ changes sign.
\end{proof}

\noindent\textbf{Strict Monotonicity of principle eigenvalue in domain.}
Consider the indefinite eigenvalue problem
\begin{equation}\label{eigen}
\begin{aligned}
  \De_p^2u &= \la g,\,\, \text{in}\, \Om ,\\
  u= &0 = \De u \,\,\, \text{on}\,\, \pa \Om,
\end{aligned}
\end{equation}
where $g(x)$ is indefinite weight function.
\begin{theorem} Let $\la_1^+(\Om)>0$ be the principle eigenvalue of \eqref{eigen}, then suppose $\Om_1 \subset \Om_2$ and $\Om_1\neq \Om_2$. Then $\la_1^+(\Omega_1)>\la_1^+(\Om_2),$ if both exist.
\end{theorem}
\begin{proof}
Let $u_i$ be a positive eigenfunction associated with $\la_1^+(\Om_i), i=1,2.$ Evidently for $\phi\in C_0^\infty(\Om_1)$,we have
\begin{align*}
0 &\leq \int_{\Om_1} L(\phi_1,u_2) dx \\
& = \int_{\Om_1} |\De \phi|^p dx-\int_{\Om_1} \frac{\phi^p}{f(u_2)}\De \left(|\De u_2|^{p-2}\De u_2\right) dx \\
& = \int_{\Om_1} |\De \phi|^p dx-\la_1^+(\Om_2)\int_{\Om_1}g(x)\phi^p dx.
\end{align*}
Letting $\phi\rar u_1$, we get
\[0\leq \int_{\Om_1} L(u_1,u_2) dx=\left(\la_1^+(\Om_1)-\la_1^+(\Om_2)\right)\int_{\Om_1} g \phi^p dx,\]
this gives $\la_1^+(\Om_1)-\la_1^+(\Om_2) >0$, as if $\la_1^+(\Om_1)=\la_1^+(\Om_2)$, we conclude that $u_1=ku_2$, which is not possible as $\Om_1 \subset \Om_2$ and $\Om_1\neq \Om_2$. This completes the proof.
\end{proof}
\noindent\textbf{Quasilinear System with Singular nonlinearity.} We will use Picone's identity to establish a linear relationship between solutions of a quasilinear system with singular nonlinearity.
Consider the singular system of elliptic equations
\begin{equation}\label{sing_sys}
\begin{aligned}
  \De_p^2u &= f(v),\,\, \text{in}\, \Om ,\\
  \De_p^2v &= \frac{\left(f(v)\right)^2}{u},\,\, \text{in}\, \Om,\\
  u>0, &\, v>0 \,\, \,\,\,\text{in}\, \Om,\\
  u= &0 = v \,\,\, \text{on}\,\, \pa \Om, \\
  \De u= &0 = \De v \,\,\, \text{on}\,\,\pa \Om,
\end{aligned}
\end{equation}
where $f$ satisfies $f'(y)\geq (p-1)[f(y)]^{\frac{p-2}{p-1}}.$
\begin{theorem}
Let$(u,v)$ be a weak solution of \eqref{sing_sys} and $f$ satisfy $f'(y)\geq (p-1)[f(y)]^{\frac{p-2}{p-1}}$, then $u=c_1v$, where $c_1$ is a constant.
\end{theorem}
\begin{proof}
Let $(u,v)$ be weak solution of \eqref{sing_sys}. Then for any $\phi_1, \phi_2\in H^2(\Om)\cap H_0^1(\Om)$,we have
\begin{equation}\label{1}
\int_\Om |\De u|^{p-2}\De u \De \phi_1 \, dx=\int_\Om f(v)\phi_1\, dx
\end{equation}
\begin{equation}\label{2}
\int_\Om |\De v|^{p-2}\De v \De \phi_2\, dx=\int_\Om \frac{f^2(v)}{u^{p-1}}\phi_2dx.
\end{equation}
Choosing $\phi_1=u$ and $\phi_2=\frac{u^p}{f(v)}$ in \eqref{1} and \eqref{2} respectively, we get
\begin{align*}
\int_\Om |\De u|^p dx &=\int_\Om uf(v) dx \\
&= \int_\Om |\De v|^{p-2} \De v\De \left(\frac{u^2}{f(v)}\right) dx,
\end{align*}
which gives
\[\int_\Om \left(|\De u|^p-\De v |\De v|^{p-2} \De \left(\frac{u^2}{f(v)}\right)\right)dx=0,\]
or \[\int_\Om R(u,v) dx=0,\]
this gives $R(u,v)=0,$ which in turn implies that $u=c_1v$.
\end{proof}
\noindent\textbf{Morse Index.}
Let us consider the problem
\begin{equation}\label{eigen_prob}
\begin{split}
\De_p^2 u=a(x)f(u),\,\,\,\,\text{in}\,\Om,\\
u=0=\De u,\,\,\,\,\,\text{on}\, \pa\Om.
\end{split}
\end{equation}
The morse index of the solution of the \eqref{eigen_prob} is the number of negative eigenvalues of the linearized operator
\[\De_p^2-a(x)f'(u)\]
acting on $H^2(\Om)\cap H_0^1(\Om)$, that is, the number of eigenvalue $\la$ such that $\la<0$ and the boundary value problem
\begin{equation}\label{lin_operator}
\begin{split}
\De_p^2w-a(x)f'(u)w=\la w,\,\,\,\text{in}\,\Om,\\
w=0=\De w,\,\,\,\,\text{on}\, \pa \Om
\end{split}
\end{equation}
has a nontrivial solution $w$ in $H^2(\Om)\cap H_0^1(\Om)$.
\begin{theorem}
Let us consider \eqref{eigen_prob}. Suppose $f'(0)\leq 1\leq f'(s)$, $\forall s\in(0,\, \infty)$ and $f(0)=0$. Let $a(x)$ be a positive continuous function in $\bar{\Om}$. Then the trivial solution of \eqref{eigen_prob} has morse index $0$.
\end{theorem}
\begin{proof}
Let $v\in H^2(\Om)\cap H_0^1(\Om)$ be a positive solution of \eqref{eigen_prob}. Then
\begin{equation}\label{3}
\int_\Om |\De v|^{p-2}\De v\De \psi\, dx=\int_\Om a(x)f(v)\psi\, dx,\,\,\,\forall \psi\in H^2(\Om)\cap H_0^1(\Om).
\end{equation}
For any $0\neq w\in H^2(\Om)\cap H_0^1(\Om)$, let us take $\psi=\frac{w^2}{f(v)}$ as a test function in \eqref{3}, we get
\begin{equation}\label{4}
\int_\Om |\De v|^{p-2}\De v\De \left(\frac{w^2}{f(v)}\right)dx=\int_\Om a(x)f(v)\frac{w^2}{f(v)}\, dx,
\end{equation}
on using  $R(u,v)\geq 0$, we get
\begin{equation}\label{5}
\int_\Om |\De w|^p dx\geq \int_\Om a(x) w^2 dx\geq \int_\Om a(x)f'(0) w^2 dx.
\end{equation}
Consider the eigenvalue problem associated with the linearization of \eqref{eigen_prob} at $0$, which is nothing but
\begin{equation}\label{linearized_eigen}
\begin{split}
\De_p^2w-a(x)f'(0)w=\la w,\,\,\,\text{in}\,\Om,\\
w=0=\De w,\,\,\,\,\text{on}\, \pa \Om
\end{split}
\end{equation}
By variational characterization of the eigenvalue in \eqref{linearized_eigen}, from \eqref{5}, we get that $\la\geq 0$ and corresponding eigenfunction is positive. Which proves the claim.
\end{proof}

\end{document}